\documentclass[letterpaper, 10 pt, conference]{cssconf}  

\IEEEoverridecommandlockouts                              

\usepackage{graphics} 
\usepackage{amsmath} 
\usepackage{amssymb}  

\usepackage{hyperref}
\usepackage{cite}

\usepackage{graphicx}
\newcommand{\citep}{\cite}
\newtheorem{theorem}{Theorem}[section]
\newtheorem{lemma}{Lemma}
\newtheorem{assumption}{Assumption}
\newtheorem{remark}{Remark}
\newtheorem{definition}[theorem]{Definition}

\title{\LARGE \bf
Impacts of Network Topology on the Performance of a Distributed Algorithm Solving Linear Equations
}

\author{Hong-Tai Cao$^{1}$, Travis E. Gibson$^{2}$, Shaoshuai Mou$^{3}$ and Yang-Yu Liu$^{4}$ 
\thanks{$^{1}$Department of Electrical Engineering, University of Southern California, Los Angeles, California 90089, USA. Channing Division of Network Medicine, Brigham and Women's Hospital, Harvard Medical School, Boston, Massachusetts 02115, USA. Email:caohongtai2014@gmail.com}%
\thanks{$^{2}$Channing Division of Network Medicine, Brigham and Women's Hospital, Harvard Medical School, Boston, Massachusetts 02115, USA. Email:tgibson@mit.edu}
\thanks{$^{3}$School of Aeronautics {\rm \&} Astronautics, Purdue University, West Lafayette, Indiana 47907, USA.  Email:mous@purdue.edu}
\thanks{$^{4}$Channing Division of Network Medicine, Brigham and Women's Hospital, Harvard Medical School, Boston, Massachusetts 02115, USA. Center for Cancer Systems Biology, Dana-Farber Cancer Institute, Boston, Massachusetts 02115, USA. Email: yyl@channing.harvard.edu}
}

\begin{document}

\maketitle
\thispagestyle{empty}
\pagestyle{empty}

\begin{abstract}
Recently a distributed algorithm has been proposed for multi-agent networks to solve a system of linear algebraic equations, by assuming each agent only knows part of the system and is able to communicate with nearest neighbors to update their local solutions. This paper investigates how the network topology impacts exponential convergence of the proposed algorithm. It is found that networks with higher mean degree, smaller diameter, and homogeneous degree distribution tend to achieve faster convergence.  Both analytical and numerical results are provided.

\end{abstract}

\section{Introduction}
A major goal in studying networked systems is to understand the impact of network topology within the context of the application of interest, from epidemic spreading \citep{pastor2001epidemic,cohen2000resilience} to synchronization \citep{nishikawa2003heterogeneity,wang2005partial}, controllability \citep{liu2011controllability,jadbabaie2004stability,pasqualetti2014controllability} , observability \citep{liu2013observability}, flocking \citep{vicsek1995novel,jadbabaie2003coordination} and consensus \citep{tsitsiklis1984problems,tsitsiklis1986distributed,murray2003consensus,olfati2007consensus}.

Recently, Mou \emph{et al.} proposed a network-based distributed algorithm to solve for $x$ in the linear equation $\mathbf{A}x=b$ \citep{mou2013fixed,mou2014distributed}. In this algorithm it is assumed that each agent is located in a communication network and has partial knowledge of $\mathbf{A}$ and $b$. Under mild conditions on the connectivity of the underlying network, all the agents' states (or local solutions) converge to the exact solution $x=\mathbf{A}^{-1}b$ \citep{mou2013fixed,liu2013asynchronous,mou2014distributed,anderson2015decentralized,mou2015distributed}.

The proposed algorithm in \citep{mou2014distributed} is distributed, applicable for all linear equations as long as they have solutions, works for time-varying networks, converges exponentially fast, operates asynchronously, and does not involve any small step-size. The aim of this paper is to further characterize the relation between its exponential convergence and the network topology. The main contribution of this work is an analytical bound that connects the convergence rate of the algorithm to the network topology and the linear equation. Both theoretical and numerical results show that networks with higher mean degree, smaller diameter, and homogeneous degree distributions tend to speed up this distributed algorithm.

The following notation is used throughout the paper. The $\ell^{2}$-norm is denoted as $\|\cdot\|$. Matrices are denoted by upper case letters in bold such as $\mathbf{A}$ and $\mathbf{P}$. A partition of a matrix is denoted by an upper case letter with a subscript, i.e. $A_i$ is a partition of matrix $\mathbf{A}$, which can also be a row vector. Vectors are denoted by lower case italic letters, such as $x$, $y$, $z$. A network or graph is denoted as $\mathcal{G}(\mathcal{V}, \mathcal{E})$, where $\mathcal{V}$ is the node (or vertex) set and $\mathcal{E}$ is the link (or edge) set. The network topology is represented by the adjacency matrix $\mathcal{A}=\{\alpha_{ij}\}$ of the network. This paper is organized as follows. The network-based distributed algorithm is briefly presented in Section \ref{sec_alg}. The theory of how the network topology impacts the algorithm performance is present in Section \ref{sec_topo}. The main proof is presented in Section \ref{sec_proof}. Finally, the conclusion is presented in Section \ref{sec_conclu}.
\\
\section{A Distributed Algorithm for Solving Linear Equations}
\label{sec_alg}
Consider a system of linear algebraic equations
\begin{equation}
\label{Ax=b}
\mathbf{A}x =b,
\end{equation} which has a unique solution $x^*$. Here $\mathbf{A} \in \mathbb{R}^{\mathit{n}\times\mathit{n}}$, $b \in \mathbb{R}^{\mathit{n}}$ and $x\in \mathbb{R}^{n}$. The partition of the matrix $\mathbf{A}$ is defined as $\mathbf{A}=\mathrm{col}\left\lbrace A_1, A_2,\cdots, A_m \right\rbrace$, where $\mathrm{col}\{\cdot\}$ is an operator that stacks elements into a column, $A_i\in \mathbb{R}^{n_i\times n}$, and the partition of the vector $b$ is defined as $b=\left[ b_1, b_2, \cdots, b_m \right]^{\mathrm{T}}$, $b_i\in \mathbb{R}^{n_i}$, where $\sum_{i=1}^{m}n_i=n$. Assume that the entire system $\left(\mathbf{A},b\right)$ is unavailable to a single agent; instead different partitions of the system $\left( A_i^{n_i \times n},b_i^{n_i}\right)$ are available to different agents. In this paper we consider the simplest case: $n_i=1$ and $m=n$, i.e. each agent knows exactly one row of $\mathbf{A}$ matrix and one element of the $b$ vector.

The distributed algorithm proposed in \citep{mou2014distributed} computes the solution of the linear equation \eqref{Ax=b} through a multi-agent network $\mathcal{G}(\mathcal{V}, \mathcal{E})$, where $\mathcal{V}=\{1,2,\cdots, n\}$ and $\mathcal{E}\subseteq \mathcal{V} \times \mathcal{V}$. The topology of this $n$-agent network is represented by its adjacency matrix $\mathcal{A}(\mathcal{G})=\left[ \alpha_{ij}\right]_{n\times n}$ with
\begin{equation*}
\alpha_{ij}= \left\lbrace
\begin{aligned}
& 1\ \mathrm{if}(i,j)\in \mathcal{E} \\
& 0\ \mathrm{otherwise.}
\end{aligned}\right.
\end{equation*} Agent $i$ in the network is synonymous with vertex $i$ in the graph $\mathcal{G}(\mathcal{V}, \mathcal{E})$. The topology of the multi-agent network is completely independent of the linear equation in \eqref{Ax=b}.

For simplicity we make the following assumption:
\begin{assumption}
The graph $\mathcal{G}$ is undirected and connected. Every vertex has a self loop and there are no multiple edges between two vertices.
\end{assumption}

Consider agent $i$ who knows $\left( A_i,b_i \right)$. It calculates its local solution $x_i\in \mathbb{R}^n$ to $A_ix_i=b_i$ and exchanges the solution $x_i$ with its neighbors, denoted as $\mathcal{N}_{i}=\{j\in \mathcal{V}|(i,j)\in \mathcal{E}\}$. In this work $t$ is the discrete time variable and takes values in $\{0,\ 1,\ 2,\cdots\}$. The exact (or global) solution to $\mathbf{A}x=b$ is obtained when all the local solutions $x_i$'s reach consensus through the following iteration procedure:
\begin{equation}
\label{x(t+1)=Mx(t)}
x_i(t+1) =x_i(t)-\frac{1}{d_i}\mathbf{P}_i\left(d_ix_i(t)-\sum_{j \in \mathcal{N}_{i}}x_j(t)\right),
\end{equation}
where $\mathbf{P}_i =\mathbf{I}-A_i^{\mathrm{T}}{\left(A_i\cdot A_i^{\mathrm{T}}\right)}^{-1}A_i$ is the orthogonal projection on the kernel of $A_i$, $i=1,\cdots,n$, and $d_i=\sum_{j=1}^n \alpha_{ij}$ is the degree of agent $i$.

Let $x^*$ be the true solution to \eqref{Ax=b} and it must satisfy $A_ix^*=b_i$ for $i=1,\cdots,n$. Define the error between $x_i(t)$ and $x^*$ as
\begin{equation}
\label{y_i}
y_i(t)=x_i(t)-x^*,
\end{equation} which is in the kernel of $A_i$. In addition, note that $\mathbf{P}_i^2=\mathbf{P}_i$ and $\mathbf{P}_iy_i(t)=y_i(t)$. Replacing $x_i(t+1)$ and $x_i(t)$ by $y_i(t+1)$ and $\mathbf{P}_iy_i(t)$ in \eqref{x(t+1)=Mx(t)}, we get the {\em error updating equation}
\begin{equation}
\label{y+=my}
y_i(t+1) =\frac{1}{d_i}\mathbf{P}_i\sum_{j \in \mathcal{N}_{i}}\mathbf{P}_jy_j(t),
\end{equation}
for $i=1,\cdots,n$. These $n$ equations can be rewritten in the following compact form
\begin{equation}
\label{y=My}
y(t) =\left(\mathbf{P}_{\mathrm{diag}}\left[\left(\mathbf{D}^{-1}\mathcal{A}^{\mathrm{T}}\right)\otimes \mathbf{I}\right]\mathbf{P}_{\mathrm{diag}}\right)^t y(0) =\mathbf{M}^t y(0),
\end{equation}
where the matrix $\mathbf{M}$ is called the {\em updating matrix} and $y(t)=\mathrm{col}\left\lbrace y_1(t),y_2(t),\cdots,y_n(t)\right\rbrace$. The matrix $\mathbf{P}_{\mathrm{diag}}=\mathrm{diag}\{\mathbf{P}_1,\mathbf{P}_2,\cdots,\mathbf{P}_n\}\in \mathbb{R}^{n^2 \times n^2}$ is a block diagonal matrix with $\mathbf{P}_i \in \mathbb{R}^{n \times n}$ and $\mathbf{D}=\mathrm{diag}\{d_1,d_2,\cdots,d_n\}$ is a diagonal matrix. The operator $\otimes$ is the kronecker product \citep{neudecker1969note}.

This algorithm has been proven to converge by using the mixed norm \citep{russo2013contraction} \citep[Chapter 4.3.1]{mou2014distributed} of $\mathbf{M}$ defined as
\begin{equation*}
\|\mathbf{M}\|_{\mathrm{mix}}=\| \mathbf{Q} \|_\infty,
\end{equation*} where $\mathbf{Q}=\{q_{ij}\}$, $q_{ij}=\frac{\alpha_{ij}}{d_i}\|\mathbf{P}_i\mathbf{P}_j\|$. Indeed, $\mathbf{M}^t$ satisfies $\lim_{t\to \infty}\|\mathbf{M}^t\|_{\mathrm{mix}}=0$ if the undirected multi-agent network is connected \citep{mou2014distributed}. Therefore $y=\mathbf{M}^ty(0)\to 0$ and thus $x_i\to x^*$ for all $i\in \mathcal{V}$.

Network properties play important roles in consensus problems. In particular, the second smallest eigenvalue $\lambda_2(\mathcal{L})$ of the graph laplacian bounds the convergence rate of consensus \citep{fiedler1973algebraic,olfati2007consensus}. Given the fact that projection matrices $\mathbf{P}_i$'s are used in constructing the updating matrix $\mathbf{M}$, it is not clear how the network topology $\mathcal{A}$ impacts the convergence rate of this algorithm. Thus, in this work we approach the proof of convergence from a different angle.

\section{Impacts of Network Topology on the Distributed Algorithm}
\label{sec_topo}
\subsection{Theoretical Analysis}
In this section, we study how network topology impacts the performance of the network-based distributed algorithm. Before we state the main theorem, we introduce the following definitions.

\begin{definition}[Walk]
In a graph $\mathcal{G}$, a walk $w^{l}\in \mathcal{V}^{l+1}$ \citep{chung1997spectral} of length $l$ is a sequence of vertices $(v_0,v_1,\cdots,v_l)$ with $\{v_{i-1},v_{i}\}\in \mathcal{E}(\mathcal{G})$ for all $1\leqslant i \leqslant l$ when $l \geqslant 1$. If $l=0$, then $w^{0}$ is simply a vertex $v_0$. Specifically, we denote a walk of length $l$ starting at vertex $v_0$ and ending at vertex $v_{l}$ as $w_{v_0 v_l}^{l}$.
\end{definition}

\begin{definition}[$f(w^l,\beta)$ Product of a Walk]
Let $w^{l}$ be a walk of length $l$. Let $\beta_{v_i}\in U$ be a value associated with vertex $v_i$. We can define a function of the walk $w^l$ as
\begin{equation*}
f(w^l,\beta)=\Pi_{i=0}^{i=l}\beta_{v_i},
\end{equation*} where $\beta$ is indexed by the walk $w^l = (v_0, v_1, \cdots, v_l)$ with values $\beta = (\beta_{v_0}, \beta_{v_1}, \cdots, \beta_{v_l})$. The function $f(w^l,\beta)\in U$ is called the {\em product of walk $w^{l}$}. In this work $U$ is either $\mathbb{R}$ or $\mathbb{R}^{n\times n}$.
\end{definition}

\begin{definition}[$\mathbb{S}(l)$ and $\mathbb{S}^{1}(l)$ Spaces]
In a graph $\mathcal{G}$, all the possible walks of length $l$ form the $\mathbb{S}(l)$ Space. Denote a subspace of $\mathbb{S}(l)$ as $\mathbb{S}^{1}(l)$ if and only if
\begin{itemize}
\item the walk $w^{l}$ starts from an arbitrary vertex $v_0$ and ends at $v_{l}$ and visits all the vertices $v_i\in \mathcal{V}$ of $\mathcal{G}$,
\item there does not exist a vertex $v_j\in \mathcal{V}$ that divides $w^{l}$ into two sub-walks, where one walk starts at $v_0$ and ends at $v_j$, the other one starts at $v_j$ and ends at $v_l$, that both of them visit all the vertices $v_i\in \mathcal{V}$ of $\mathcal{G}$.
\end{itemize}
Note that the end vertex of the previous sub-walk and the starting vertex of the following sub-walk are repeated twice when dividing a walk. It is trivial that for $w^{l}$ walks of length $l\leqslant n-1$, they can't be in the $\mathbb{S}^{1}(l)$ subspace.
\end{definition}

\begin{definition}[Order $r$]
If a walk $w^{l}$ can be divided into several walks $w^{l_1}$, $w^{l_2}$, $\cdots$, $w^{l_r}$, where $l_i\geqslant 1$ and $w^{l_i}\in \mathbb{S}^{1}(l_i)$, then all the walks of the same number $r$ form a subspace $\mathbb{S}^{r}(l)$ where $r$ is called the {\em order} of the space. We also say that $r$ is the {\em order} of the walk $w^{l}$. $\mathbb{S}^r(l) \subsetneq\mathbb{S}(l)$ for any order $r$.

If a walk $w^l$ does not visit all the vertices in a graph $\mathcal{G}$, then its order is $r=0$ and it is in $\mathbb{S}^{0}(l)$. This special case means that there exists at least one vertex $v_i\in \mathcal{V}$ which does not appear in the sequence of the walk $w^{l}$. The order of any $w^{l}$ walk is uniquely determined and non-negative, i.e. $r\geqslant 0$.
\end{definition}

Let $\varphi=\frac{1}{\left(\sqrt{n}\tau\|\mathbf{A}^{-1}\|\right)^2}$, $\tau=\underset{i}{\max}\left(\|A_i\|\right)$,  $\frac{1}{d}=\left(\frac{1}{d_{i}},\frac{1}{d_{v_1}},\cdots,\frac{1}{d_j}\right)$ be indexed by the walk $w_{ij}^{t}=\left(i,v_1,\cdots,v_{t-1},j\right)$ which starts at agent $i$ and ends at agent $j$ where $w_{ij}^{t}\in \mathcal{V}^{t+1}$, then we have the following theorem
\begin{theorem}[Convergence Bound]
\label{th_bound}
Given a linear equation $\mathbf{A}x=b$, $\mathbf{A}=\mathrm{col}\{A_i\}\in \mathbb{R}^{n \times n}$ and its unique solution $x^*$, let $x_{i}(t)$ be the local solution at agent $i$ located in an undirected network $\mathcal{G}(\mathcal{V},\mathcal{E})$ whose adjacency matrix is $\mathcal{A}=\{\alpha_{ij}\}$, then the error $y_{i}(t)$ defined in \eqref{y_i} is bounded as
\begin{equation}
\label{bound_of_th}
\|y_{i}(t+1)\| \leqslant \sum_{\mathcal{N}_j}\sum_{r=0}^{r_m(t)}\sum_{w_{ij}^{t}\in \mathbb{S}^{r}} f(w_{ij}^t,\frac{1}{d})\left(1-\varphi \right)^{\frac{nr}{2}} \|y_j(0)\|
\end{equation} for $i=1,\cdots, n$. Here $r_m(t)\leqslant \lfloor\frac{t}{n}\rfloor$ is the maximum order of the product. Note that $w_{ij}^{0}=w_{ii}^{0}=\left(i\right)$ and $w_{ij}^{1}=\left(i,j\right)$.
\end{theorem}


Theorem \ref{th_bound} provides another method to prove that the distributed algorithm converges to the true solution $x^*$ besides the mixed norm method in \citep{mou2014distributed}, which is discussed at the end of this work. The bound in \eqref{bound_of_th} connects the network topology with the convergence rate of the algorithm, by the degree $d_i$ of agent $i$ explicitly, and by counting the number of $w^t\in \mathbb{S}^{r}(t)$ walks in every order $r\geqslant 0$ in the network implicitly. Before moving to the detailed proof of this theorem, we first discuss how topology impacts the performance of the algorithm. To illustrate the topology impacts, we start with the definition of a walk $w^{t}$, then we discuss the properties of the corresponding ${f}(w^t,\frac{1}{d})$ product.

Given a network $\mathcal{G}$ of size $n$, all the possible walks of length $t$ are determined by its adjacency matrix $\mathcal{A}=\{\alpha_{ij}\}$. Let $\frac{1}{d_i}$ be the inverse degree of agent $i$, then the product $\frac{1}{d_{i_0}}\frac{1}{d_{i_1}}\cdots\frac{1}{d_{i_t}}$ can be represented by ${f}^{r}(w_{i_0i_t}^t,\frac{1}{d})$, where we recall that $\frac{1}{d}$ is indexed by the walk $w_{i_0i_t}^{t}$. For simplicity, we let $i=i_0$ and $j=i_t$. Hence given a starting agent $i$, the summation of all products of the walk $w^{1}$ from $i$ to all the agents $j=1,2,\cdots,n$ is represented as $\sum_{j=1}^{n}\frac{\alpha_{ij}}{d_{i}d_{j}}$. In general, we have
\begin{equation*}
\begin{aligned}
\sum_{r=0}^{r_m(t)}\sum_{w_{ij}^{t}} {f}^{r}(w_{ij}^t,\frac{1}{d}) = \sum_{l_{t-1}=1}^{n}\cdots\sum_{l_{1}=1}^{n}\frac{\alpha_{il_1}\alpha_{l_1l_2}}{d_id_{l_1}}\cdots\frac{\alpha_{l_{t-1}j}}{d_{l_{t-1}}}\frac{1}{d_{j}}.
\end{aligned}
\end{equation*} It is trivial that for any $r$, $i$, $j$ and the walk $w_{ij}^t$, $f(w_{ij}^t,\frac{1}{d})\in (0,1)$. We now explore a scenario when the above mentioned sum remains a constant, even if the walk length increases.

Given a network $\mathcal{G}$ and given a starting agent $i$, if all walks $w_{ij}^{t}$, $j=1,2,\cdots,n$ are repeated by walks $w_{ij^{\prime}}^{t+1}$ who visit one more agent $j^{\prime}$ at the end, after reaching agent $j$, then the summation of all ${f}(w_{ij^{\prime}}^{t+1},\frac{1}{d})$ products remains the same. This visit of agent $j^{\prime}$ generates $n$ products based on each $f(w_{ij}^t,\frac{1}{d})$ and each of them equals to $\frac{\alpha_{jj^{\prime}}}{d_{j^{\prime}}}f(w_{ij}^t,\frac{1}{d})$, $j=1,2,\cdots,n$. Only $d_{j}$ out of $n$ products are not zero when $\alpha_{jj^{\prime}}=1$. The summation of all newly generated products is unchanged, which is
\begin{equation}
\label{sum_visit} 
\sum_{\mathcal{N}_{j^{\prime}}} \frac{1}{d_{j^{\prime}}} \sum_{\mathcal{N}_{j}} {f(w_{ij}^t,\frac{1}{d})} = \sum_{\mathcal{N}_{j}} f(w_{ij}^t,\frac{1}{d})
\end{equation} for $\sum_{\mathcal{N}_{j^{\prime}}} =d_{j^{\prime}}$. In general, the summation of all products of all walks by $t+1$ visits starting from a given agent $i$ to all the neigbors of all the agents $j$ is
\begin{equation}
\label{sum_walk}
\begin{aligned}
& \sum_{\mathcal{N}_j}\sum_{r=0}^{r_m(t)}\sum_{w_{ij}^{t}\in \mathbb{S}^{r}(t)} f(w_{ij}^t,\frac{1}{d}) \\
= & \sum_{j^{\prime}=1}^{n}\sum_{j}^{n}\sum_{i_{t-1}}^{n}\cdots \sum_{i_{1}}^{n}\frac{\alpha_{ii_{1}}}{d_{i}}\cdots \frac{\alpha_{i_{t-1}j}}{d_{i_{t-1}}}\frac{\alpha_{jj^{\prime}}}{d_{j}} = 1. \\
\end{aligned}
\end{equation}

Given a network $\mathcal{G}$ and a starting agent $i$, the summation $\sum_{\mathcal{N}_j}\sum_{w_{ij}^{t}\in \mathbb{S}^{0}(t)}f(w_{ij}^t,\frac{1}{d})$ is never increasing and the order $r$ of the $f(w^t,\frac{1}{d})$ product is never decreasing as the walk length $t$ grows. Given an arbitrary $f(w_{i_0i_t}^t,\frac{1}{d})$ product of the walk $w_{i_0i_t}^t\in \mathbb{S}^{0}(t)$, when the walk $w_{i_0i_t}^t$ makes one more visit from agent $i_{t}$ to the next agent $i_{t+1}$, it forms $d_{i_{t}}$ new products and the summation of all $d_{i_{t}}$ products is unchanged, which is already shown in \eqref{sum_visit}. However, there exists a walk of length $t_1$ when there exists at least one walk changing from the $\mathbb{S}^{0}(t_1)$ subspace to the $\mathbb{S}^{0}(t_1+1)$ subspace. For every $w^{t_2}$ walk (of order $r\geqslant 1$) of length $t_2$, it never changes to a walk of order $r=0$. This hold for any walk $w^t\in \mathbb{S}^{0}(t)$, hence the summation of all ${f}(w^t,\frac{1}{d})$, $w^t\in \mathbb{S}^{0}(t)$ product is never increasing, that is
\begin{equation*}
 \sum_{\mathcal{N}_j}\sum_{w_{ij}^{t+1}\in \mathbb{S}^{0}(t+1)}f(w_{ij}^{t+1},\frac{1}{d}) \leqslant  \sum_{\mathcal{N}_j} \sum_{w_{ij}^{t}\in \mathbb{S}^{0}(t)}f(w_{ij}^t,\frac{1}{d})
\end{equation*} and given a walk of length $t$ and a starting agent $i$, the bound in \eqref{bound_of_th} decreases when the order of walks increases, due to the exponential factor $\lim_{r\to \infty}\left(1-\varphi \right)^{\frac{nr}{2}}=0$. Since the summation of all $f$ products starting from a chosen agent $i$ is always $1$ \eqref{sum_walk}, the bound in \ref{th_bound} can only be decreased by either i) for a fixed length $t$, increasing the percentage of walks with higher $r$, or ii) by increasing the order $r$ for all walks as rapidly as possible.

With the above two observations we conclude that given any two networks $\mathcal{G}_1$ and $\mathcal{G}_2$, the distributed algorithm \eqref{x(t+1)=Mx(t)} tends to converge faster on networks $\mathcal{G}_1$ if $\mathcal{G}_1$ and $\mathcal{G}_2$ have similar topology properties except any combinations of the following
\begin{itemize}
\item[1] $\mathcal{G}_1$ has a shorter diameter,
\item[2] $\mathcal{G}_1$ has a more homogeneous degree distribution,
\item[3] $\mathcal{G}_1$ has a higher mean degree.
\end{itemize}

Although Theorem \ref{th_bound} has $\frac{1}{d}$ as a factor in the products, it is not trivial to conclude that higher degree makes the products smaller since higher degree decreases each product while increases the number of  products. The summation of all products remains a constant, as shown in \eqref{sum_walk}. However the bound decreases when the order $r$ of the products increases. We address these three points in order.

\subsubsection{Diameter}
For two graphs $\mathcal{G}_1$ and $\mathcal{G}_2$ with the same degree distribution and hence the same mean degree, if $\mathcal{G}_1$ has a shorter diameter \citep{diestel2005graph} than $\mathcal{G}_2$, then for fixed $t$, walks from $\mathcal{G}_1$ will necessarily have a larger minimum order $r$ as compared to those from $\mathcal{G}_2$. This follows from the fact that all the agents can be visited with fewer steps in a network with shorter diameter. Thus, all things being equal between two graphs, if $r(t)$ increases more rapidly for one graph as opposed to another, the exponential factor $\left(1-\varphi \right)^{\frac{nr}{2}}$  will decrease more rapidly. Therefore networks with shorter diameter make the distributed algorithm converge faster.

\subsubsection{Degree Distribution}
Let $\mathcal{G}_1$ and $\mathcal{G}_2$ be two graphs with same mean degree but different degree distributions. Let $\mathcal{G}_1$ have a more homogeneous degree distribution than $\mathcal{G}_2$. Walks in $\mathcal{G}_2$ typically have lower order $r$ than the walks of the same length in $\mathcal{G}_1$. This is because walks on $\mathcal{G}_2$ rather than $\mathcal{G}_1$ have to walk though the high degree vertices again and again to reach all the other low degree vertices. Hence for a given length of walks, the order $r$ from the walks on $\mathcal{G}_1$ is higher. Therefore homogeneous degree distribution makes the algorithm converges faster.

\subsubsection{Mean Degree}
Adding edges to a graph typically results in a shorter diameter. Given two graphs $\mathcal{G}_1$ and $\mathcal{G}_2$ with similar degree distribution where $\mathcal{G}_1$ has a higher mean degree, the diameter of $\mathcal{G}_1$ is typically no larger than $\mathcal{G}_2$. Hence the orders $r$'s from $\mathcal{G}_1$ are typically higher than those in $\mathcal{G}_2$ for walks of fixed length. Adding a new edge can either make the degree distribution homogeneous or make it heterogeneous, depending on where the new edge is added. The overall change of degree distribution for each newly added edge is difficult to analyze. However, if multiple new edges are added uniformly to a graph, this will typically result in a more homogeneous degree distribution, thus increasing the mean degree of the network makes the distributed algorithm converge faster.

\subsection{Simulation Results}
To verify our theoretical predictions, we perform extensive numerical simulations. We first quantify the convergence rate of the network-based distributed algorithm. One measure is the solution accuracy of the algorithm, which is the Euclidean distance between the local solution and the exact (or global) one:
\begin{equation*}
\epsilon_{i}(t)=\|x_i(t)-x^*\|,\ i=1,2,\cdots,n.
\end{equation*} Smaller $\epsilon_{i}$ means faster convergence rate and hence better algorithm performance. The impacts of different network topologies are measured by the statistical performances of the distributed algorithm, i.e. $E\left(\sum_{i=1}^n\epsilon_{i}\right)$ on an ensemble of linear equations. We notice that the Euclidean distance defined above needs a reference. For example, if the true solutions of two cases are $\|x^{*,1}\|=100$ and $\|x^{*,2}\|=0.1$ respectively, while the summation of Euclidean distances of all local solutions to $x^{*,j}$ are both $\sum_{i=1}^{n}\epsilon_{i}^{j}=\sum_{i=1}^{n}\|x_i^{j}-x^{*,j}\|=1$, $j=1,2$, it is obvious the accuracy of the former iterative process is much higher than the latter one. Therefore the Euclidean distance should be scaled by the initial error $\sum_{i=1}^n \epsilon_{i}(0)$, yielding the relative error
\begin{equation}
\label{rel_err}
R(t)=\frac{\sum_{i=1}^n \epsilon_{i}(t)}{\sum_{i=1}^n \epsilon_{i}(0)}=\frac{\sum_{i=1}^n \|x_i(t)-x^*\|_2}{\sum_{i=1}^n \|x_i(0)-x^*\|_2}.
\end{equation}
In this way, convergence performances among a system of linear equations can be compared.

\begin{figure}
\includegraphics[angle=-90,width=0.5\textwidth]{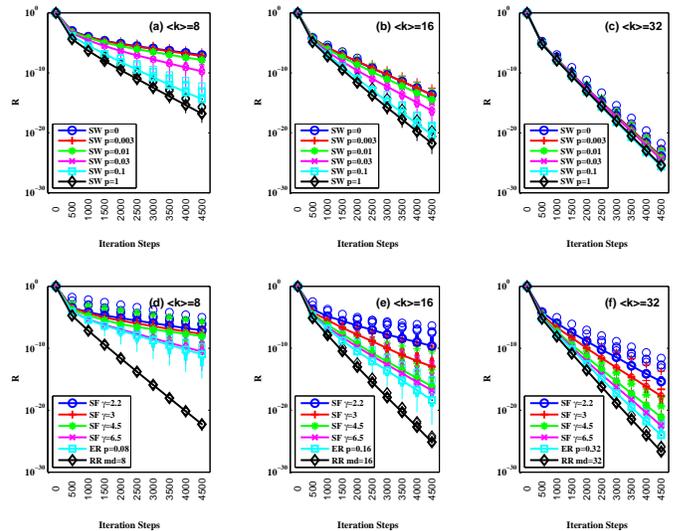}
\caption{\label{fig_cv_box}Impact of network topology on the performance of the network-based distributed algorithm. . Tens of different linear equations are solved by the distributed algorithm on six groups of networks of size $n=100$. The complex networks in each group are (a-c) Small-world (SW) networks; (d-f) Scale-free (SF) networks, Erd\"{o}s-R\'{e}nyi (ER) random graphs, random regular (RR) graphs. In each case, we show the box-and-whisker plots and the median value of the relative error (or convergence rate) $R(t)$ as functions of $t$. At each marked iteration step $t$, a box-and-whisker plot is drawn. The mean degree of the complex networks is represented as $\langle k \rangle$.}
\end{figure}

Figure.~\ref{fig_cv_box} shows the relative error changes with different network topologies, including small-world (SW) networks \citep{watts1998collective} with random rewiring probability $p$, scale-free (SF) networks \citep{barabasi1999emergence} with degree exponent $\gamma$, Erd\"{o}s-R\'{e}nyi (ER) random graphs  \citep{erdos1960evolution} with connectivity probability $p$ and random regular (RR) graphs \citep{wormald1999models} with mean degree $\langle k \rangle$. The networks in each subfigure are the same in their mean degree and they are different on only one parameter. Small-world networks (a-c) are different in rewiring probabilities $p$, which determines network diameters. Scale-free networks, graphs and RR graphs are drastically different in their degree distributions: scale-free networks are most heterogeneous and random regular graphs are most homogeneous.

\begin{figure}
\includegraphics[angle=-90,width=0.5\textwidth]{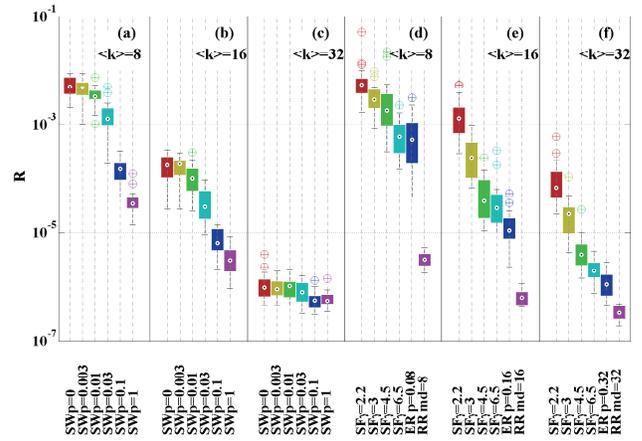}
\caption{\label{fig_cv_last_all}Convergence rate at a chosen time step for complex networks with different topologies. The box-plot shows the relative errors at a given step $T_s=2000$. Networks with similar topological features are grouped together in a particular subfigure.}
\end{figure}

The numerical results shown in Figure.~\ref{fig_cv_box} clearly verify our theoretical predictions, i.e. if two networks share similar topological properties, the one with smaller diameter (or more homogeneous degree distribution, or higher mean degree) perform better than the other. To further demonstrate the topology impacts, consider $R(t)$ at $t=2000$ shown as box-and-whisker plots in Figure.~\ref{fig_cv_last_all}. The smaller relative error $R(t)$ means higher convergence rate. It is clear from Figure.~\ref{fig_cv_last_all}a-c and Figure.~\ref{fig_cv_last_all}d-f that the upper bound of relative errors decreases as the mean degree increases for a given network model. In other words, higher mean degree makes the algorithm reach the true solution faster, and is consistent with our theoretical analysis. Figure.~\ref{fig_cv_last_all}a-c display that small-world networks with higher rewiring probability (and hence smaller diameters) have smaller relative errors $R$ , confirming our theoretical prediction smaller diameter contributes to higher convergence rate. As shown in Figure.~\ref{fig_cv_last_all}d-f, for any given mean degree, the random regular graphs have the smallest relative errors while scale free networks perform the worst. This means that the degree heterogeneity degrades the performance of the network-based distributed algorithm in solving linear equations \eqref{Ax=b}.

\section{Proof of the Bound Theorem}
\label{sec_proof}
Before the formal proof of Theorem \ref{th_bound}, we discuss the structure of the matrix $\mathbf{M}^t$ \eqref{y=My} and introduce some technical lemmas.

Let $m_{ij}^{(1)}\in \mathbb{R}^{n \times n}$ be the $i,j$-th partition matrix of $\mathbf{M}$, then
\begin{equation*}
m_{ij}^{(1)}=\frac{\alpha_{ij}}{d_i} \mathbf{P}_i \cdot \mathbf{P}_j,
\end{equation*} where we recall that $\mathbf{P}_i$ is an orthogonal projection matrix defined right after \eqref{x(t+1)=Mx(t)}. Theses block matrices $m_{ij}^{(1)}$ are actually the updating matrix of $y_i(t)$, which means $y_i(t+1)=\sum_{j=1}^{n}m_{ij}^{(1)}y_j(t)$. Similarly, let $m_{ij}^{(t)}$ denote the partition matrix of $\mathbf{M}^t$, then
\begin{equation*}
\begin{aligned}
m_{ij}^{(t)}
& = \sum_{l_{t-1}=1}^n \cdots \sum_{l_{1}=1}^n m_{il_1} \cdots m_{l_{t-1}j} \\
& = \sum_{l_{t-1}=1}^n \frac{\alpha_{l_{{t-1}}j}}{d_{l_{t-1}}}\cdots \sum_{l_{1}=1}^n \frac{\alpha_{il_{1}}\cdot \alpha_{l_{1}l_{2}}}{d_i\cdot d_{l_{1}}} \mathbf{P}_i \cdots  \mathbf{P}_{l_{t-1}} \mathbf{P}_j.
\end{aligned}
\end{equation*} Although the expression of $m_{ij}^{(t)}$ is long, it shows that $\mathbf{M}^t$ is simply a weighted sum of projection products. It follows that \eqref{y+=my} can be written as $y_i(t)=\sum_{j=1}^n m_{ij}^{(t)}y_j(0)$. Define $\mu_{ij}=\frac{\alpha_{ij}}{d_i}\in \left[0,0.5\right]$, then we have
\begin{equation}
\label{y=mu_p_y}
y_{i}(t) = \sum_{j=1}^n \cdots \sum_{l_{1}=1}^n \mu_{il_{1}} \cdots \mu_{l_{t-1}l_{j}} \mathbf{P}_{i} \cdots  \mathbf{P}_{l_{t-1}}  \mathbf{P}_j y_j(0).
\end{equation} Note that it is a summation of $n^t$ products. We now separate $\mu_{il_{1}} \mu_{l_{1}l_{2}} \cdots \mu_{l_{t-1}l_{j}} \mathbf{P}_{i}  \mathbf{P}_{l_{1}} \cdots \mathbf{P}_{l_{t-1}}  \mathbf{P}_jy_j(0)$ into a {\em $\mu$ product} 
\begin{equation}
\label{lmdlmd}
\mu_{il_{1}} \mu_{l_{1}l_{2}} \cdots \mu_{l_{t-1}j}
\end{equation} and its corresponding projection product with $y_j(0)$, which is called {\em error sequence},
\begin{equation}
\label{pppy}
\mathbf{P}_{i}  \mathbf{P}_{l_1} \cdots \mathbf{P}_{l_{t-1}}  \mathbf{P}_j y_j(0).
\end{equation} From \eqref{sum_visit} the summation of all $\mu$ products \eqref{lmdlmd} satisfies the following equality
\begin{equation}
\label{sum_lmd}
\sum_{j=1}^n\sum_{l_{t-1}=1}^n \cdots \sum_{l_{1}=1}^n \mu_{il_{1}} \mu_{l_{1}l_{2}} \cdots \mu_{l_{t-1}j} = 1.
\end{equation}

The construction of $\mathbf{M}^t$ as a $\mu$ product and an error sequence of projections allows us to separate the topological features from the part of the algorithm that is specific to a particular linear equation. We first analyse each product in the error updating equation \eqref{y=mu_p_y} by bounding the error sequences of \eqref{pppy}.

Define a sequence of vectors $z(t)\in \mathbb{R}^{n}$ as following
\begin{equation}
\label{kacz}
z^{(j)}(t+1)=z(t)+\frac{b_j-A_jz(t)}{\|A_j\|^2}A_j^{\mathrm{T}},
\end{equation} where $t\geqslant 0$ and the superscript $(j)$ corresponds to its row vector $A_j$ and its scaler $b_j$. Then
\begin{equation*}
\begin{aligned}
\mathbf{P}_{i}\left(z(0)-x^*\right)
& = z(0)-\frac{A_iz(0)}{\|A_i\|^2}A_i^{\mathrm{T}}-x^*+ \frac{b_j}{\|A_i\|^2}A_i^{\mathrm{T}}\\
& = z(0)+\frac{b_i-A_iz(0)}{\|A_i\|^2}A_i^{\mathrm{T}}-x^* \\
& = z^{(i)}(1)-x^*.
\end{aligned}
\end{equation*} Let $z^{\left(j\right)}(0)=x_j(0)$, then each error sequence in \eqref{pppy} can be written as
\begin{equation}
\label{pppy=z-x}
\begin{aligned}
  & \mathbf{P}_{i}  \mathbf{P}_{l_1} \cdots \mathbf{P}_{l_{t-2}}  \mathbf{P}_{l_{t-1}}  \mathbf{P}_j y_j(0) \\
= & \mathbf{P}_{i}  \mathbf{P}_{l_1} \cdots \mathbf{P}_{l_{t-2}} \mathbf{P}_{l_{t-1}} \left(z^{(j)}(0)-x^*\right) \\
= & \mathbf{P}_{i}  \cdots \mathbf{P}_{l_{t-2}} \left(z^{(j)}(0)+\frac{b_{l_{t-1}}-A_{l_{t-1}}z^{(j)}(0)}{\|A_{l_{t-1}}\|^2}A_{l_{t-1}}^{\mathrm{T}}-x^* \right) \\
= & \mathbf{P}_{i}  \cdots \mathbf{P}_{l_{t-2}} \left(z^{(jl_{t-1})}(1)-x^*\right) \\
= & z^{(il_{1}\cdots l_{t-2}l_{t-1}j)}(t)-x^*.
\end{aligned}
\end{equation}

Essentially, $z^{(il_{1}\cdots l_{t-2}l_{t-1}j)}(t)$ forms the sequence of $z(t)$ by taking different combinations of orthogonal projection $\mathbf{P}_i$ at different agents, $i=1,2,\cdots,n$. We now show that sequences $z(t)$ can be bounded, so that the error sequence is bounded as well.

We now present two theorems for bounding $z(t)-x^*$, first for the case when the walk $w^t$ is associated with the product $f\left(w^{t}, \mathbf{P}_i\right)$, $w^{t}\in \mathbb{S}^0(t)$, and second for the $f\left(w^{t}, \mathbf{P}_i\right)$ product where $w^{t}\in S^r(t)$ and $r\geqslant 1$.

\begin{theorem}[${f}^{0}$ Bound]
\label{th_Sbar}
For any $w^t\in \mathbb{S}^{0}(t)$ it follows that $\|f(w^t,\mathbf{P})\|\leqslant 1$ and thus $\|f(w^t,\mathbf{P})\|\leqslant 1$. Therefore the dynamics in \eqref{kacz} satisfy the following inequality
\begin{equation}
\|z(t)-x^*\| \leqslant \|z(0)-x^*\|
\end{equation}
\end{theorem}

\begin{proof}
Given that $\mathbf{P}_i$ is a normalized projection matrix it follows that $\|\mathbf{P}_i\|=1$. 
\end{proof}

\begin{theorem}[${f}$ Bound]
\label{th_S}
The sequence $z(t)-x^*$ of the part whose $\mathbf{P}_{i}  \mathbf{P}_{l_1} \cdots \mathbf{P}_{l_{t-1}}  \mathbf{P}_j$ product is an ${f}\left(w^{t},\mathbf{P}\right)$ product where $w^{t}\in \mathbb{S}^{r}(t)$ and $r\geqslant 1$, then all the sequence $\mathbf{P}_{i_1}\mathbf{P}_{i_2}\cdots y_j(0)$ in this part from \eqref{pppy=z-x} can be written as
\begin{equation*}
z(t)-x^* = f\left(w^{t},\mathbf{P}\right)y_j(0),
\end{equation*} where $z(t)-x^*$ consists of several $f(w^{i},\mathbf{P})$, $w^{i}\in \mathbb{S}^{1}(i)$ products. Then all the sequences $z(t)-x^*$ in this part are bounded by
\begin{equation*}
\begin{aligned}
\|z(t)-x^*\|
& \leqslant  \left( 1-\frac{1}{\left(\sqrt{n}\tau\|\mathbf{A}^{-1}\|\right)^2} \right)^{\frac{nr}{2}} \|z(0)-x^*\| \\
& < \left(1-\kappa(\mathbf{A})^{-2}\right)^{ \frac{nr}{2} } \|z(0)-x^*\|,
\end{aligned}
\end{equation*} where $\kappa(\mathbf{A})=\|\mathbf{A}\|\cdot\|\mathbf{A}^{-1}\|$ is the usual condition number of $\mathbf{A}$ and we recall the definition $\tau=\underset{i}{\max}\left(\|A_i\|\right)$.
\end{theorem}

The proof of Theorem \ref{th_S} requires several technical Lemmas.

\begin{lemma}[Orthogonal Projection]
\label{lm_OP}
Let $z(t)\in \mathbb{R}^{n}$, $\|z(0)\|=0$ be a sequence that follows
\begin{equation*}
z^{(j)}(t+1)=z(t)+\frac{b_j-A_jz(t)}{\|A_j\|^2}A_j^{\mathrm{T}},
\end{equation*} where $A_j$, $b_j$ are defined as those in linear equation \eqref{Ax=b}, which is the same as \eqref{kacz}. Then the orthogonal projection matrix $\mathbf{P}_{i}^{\star}$ onto the solution space of the linear equation \eqref{Ax=b} is given in \citep{strohmer2009randomized} as
\begin{equation*}
z(t+1)=\mathbf{P}_i^{\star}z(t).
\end{equation*}

Let $\langle z(t+1),z(t)\rangle$ denotes the inner product of two vectors $z(t+1)$ and $z(t)$, then the above equation can be written as follows by using the updating function \eqref{kacz}
\begin{equation*}
\begin{aligned}
\mathbf{P}_i^{\star}z(t)
& =z(t)-\frac{A_iz(t)-b_i}{\|A_i\|^2}A_i^{\mathrm{T}}\\
& =z(t)-\frac{A_iz(t)-A_iz^*}{\|A_i\|}\frac{A_i^{\mathrm{T}}}{\|A_i\|}\\
& =z(t)-\langle z(t)-z^*,Z_i\rangle Z_i^{\mathrm{T}},
\end{aligned}
\end{equation*}
where $Z_i=\frac{A_i}{\|A_i\|}$, $i=1,2,\cdots, n$, $\|Z_i\|=1$ is a set of normal vectors in the hyperplane $\{z(t): \langle A_i, z(t)\rangle=b_i\}$.
\end{lemma}

\begin{lemma}[Orthogonality]
\label{lm_orthogonality}
Consider the linear equation \eqref{Ax=b} and let $x^*$ be the unique solution. The difference of two vectors $z(t+1)$ and $z(t)$ is in the kernel of $\mathbf{P}_i^{\star}$ by Orthogonal Projection Lemma \ref{lm_OP}, which means that it is orthogonal to the solution space. Therefore it is also orthogonal to $z(t+1)-x^*$. In other words, the orthogonality of two vectors $z(t+1)-z(t)$ and $z(t+1)-x^*$ satisfies
\begin{equation*}
\|z(t+1)-z(t)\|^2+\|z(t+1)-x^*\|^2=\|z(t)-x^*\|^2.
\end{equation*}
\end{lemma}

\begin{lemma}[Inequality]
\label{lm_inequal}
Let $\mathbf{A}=\mathrm{col}\{A_i\}$, $\mathbf{A}\in \mathbb{R}^{n\times n}$ is full rank. Then the following inequality holds
\begin{equation*}
\sum_{i=1}^n \|\langle \frac{A_i}{\|A_i\|}, x \rangle\|^2 \geqslant \frac{1}{\left(\tau\|\mathbf{A}^{-1}\|\right)^2}\|x\|^2.
\end{equation*} where $\langle A_i, x \rangle$ denotes the inner product of vector $A_i$ and $x$ and we recall the definition $\tau=\underset{i}{\max}\left(\|A_i\|\right)$.
\end{lemma}

\begin{proof}[Proof of Inequality Lemma \ref{lm_inequal}]
Consider the linear equation in \eqref{Ax=b} and using the submultiplicative property of the $\ell^2$-norm the following holds
\begin{equation*}
\|\mathbf{A}^{-1}\|^2\cdot\|\mathbf{A}x\|^2 \geqslant \|\mathbf{A}^{-1}\mathbf{A}x\|^2,\ \forall\ x\in\mathbb{R}^{n},
\end{equation*} where $\mathbf{A}^{-1}$ is defined because $x^*$ is the unique solution of the linear equation in \eqref{Ax=b}.
Considering the matrix partition $\mathbf{A}=\mathrm{col}\{A_i\}$, we have
\begin{equation*}
\sum_{i=1}^n \|\langle A_i, x \rangle\|^2
= \sum_{i=1}^n \|A_i\|^2\|\langle \frac{A_i}{\|A_i\|}, x \rangle\|^2
\geqslant \frac{\|x\|^2}{\|\mathbf{A}^{-1}\|^2}.
\end{equation*} Moreover,
\begin{equation*}
\begin{aligned}
\sum_{i=1}^n \tau^2\|\langle \frac{A_i}{\|A_i\|}, x \rangle\|^2
& \geqslant \sum_{i=1}^n \|A_i\|^2\|\langle \frac{A_i}{\|A_i\|}, x \rangle\|^2 \\
& \geqslant \frac{1}{\|\mathbf{A}^{-1}\|^2}\|x\|^2,
\end{aligned}
\end{equation*} where $\tau>0$ since $\mathbf{A}$ is full rank. Dividing by $\tau$ we arrive at the following inequality
\begin{equation*}
\sum_{i=1}^n \|\langle \frac{A_i}{\|A_i\|}, x \rangle\|^2 \geqslant \frac{1}{\left(\tau\|\mathbf{A}^{-1}\|\right)^2}\|x\|^2.
\end{equation*}
\end{proof}

\begin{proof}[Proof of ${f}$ Bound \ref{th_S}]
Let $x^*$ denote the unique solution to the linear equation \eqref{Ax=b}. Let $z(t)-x^*$ be vector sequence from the $f(w^{t},\mathbf{P}_i)$, $w^{t}\in \mathbb{S}^{r}(t)$ product part of the error sequence \eqref{pppy=z-x} where $r\geqslant 1$ and substitute the $z(t+1)$ by the updating function \eqref{kacz} in the the Orthogonality Lemma \ref{lm_orthogonality} then we have
\begin{equation*}
\begin{aligned}
& \|z(t+1)-x^*\|^2 \\
= & -\|z(t+1)-z(t)\|^2+\|z(t)-x^*\|^2 \\
= & -\|\frac{\langle A_i,z(t)-x^*\rangle}{\|A_i\|}\frac{A_i^{\mathrm{T}}}{\|A_i\|}\|^2+\|z(t)-x^*\|^2 \\
= & -\|\langle z(t)-x^*,Z_i \rangle\|^2 +\|z(t)-x^*\|^2,
\end{aligned}
\end{equation*} where $Z_i=\frac{A_i}{\|A_i\|}$. Since the walk $w^{t}\in \mathbb{S}^{r}(t)$, $r\geqslant 1$, the subscript $i$ in $Z_i=\frac{A_i}{\|A_i\|}$ takes all the values $1,2,\cdots,n$ at least once. There exists $\theta_{i(t)}\geqslant 0$ such that
\begin{equation*}
\|\langle z(t)-x^*,Z_{i} \rangle\|^2 \geqslant \frac{\theta_{i(t)}}{\left(\tau\|\mathbf{A}^{-1}\|\right)^2}\|z(t)-x^*\|^2
\end{equation*} for ${i(t)}=1,2,\cdots,n$, by the Inequality Lemma \ref{lm_inequal}. Note that
\begin{equation*}
\begin{aligned}
\frac{\theta_{i(t)}}{\left(\tau\|\mathbf{A}^{-1}\|\right)^2}\|z(t)-x^*\|^2 
& \leqslant \|\langle z(t)-x^*,Z_{i} \rangle\|^2 \\
& \leqslant \|z(t)-x^*\|^2,
\end{aligned}
\end{equation*} where $\|Z_{i}\|=1$. Therefore $\frac{\theta_{i(t)}}{\left(\tau\|\mathbf{A}^{-1}\|\right)^2}\leqslant 1$ for $i(t)=1,\cdots, n$, and then $\|z(t)-x^*\|^2$ is bounded as
\begin{equation*}
\begin{aligned}
& \|z(t)-x^*\|^2 \\
\leqslant & \left(1-\frac{\theta_{i(1)}}{\left(\tau\|\mathbf{A}^{-1}\|\right)^2}\right) \cdots \left(1-\frac{\theta_{i(t)}}{\left(\tau\|\mathbf{A}^{-1}\|\right)^2}\right) \|z(0)-x^*\|^2,
\end{aligned}
\end{equation*} where
\begin{equation}
\label{>=0}
0 \leqslant \left(1-\frac{\theta_{i(t)}}{\left(\tau\|\mathbf{A}^{-1}\|\right)^2}\right) \leqslant 1.
\end{equation}

Note that the sequence $z(t)-x^*$ forms the $f(w^{t},\mathbf{P})$, $w^{t}\in \mathbb{S}^{r}(t)$, $r\geqslant 1$ product part. Because of the fact $\mathbf{P}_{i}^{r}=\mathbf{P}_{i}$, all ${i(t)}=1,2,\cdots,n$ are present at least once in the each sub-walk of the original walk by definition. Hence the walk $w^{t}$ corresponding to $\left(1-\frac{\theta_{i(1)}}{\left(\tau\|\mathbf{A}^{-1}\|\right)^2}\right) \cdots \left(1-\frac{\theta_{i(t)}}{\left(\tau\|\mathbf{A}^{-1}\|\right)^2}\right)$ is divided into $r$ sub-walks $w^{t_i}\in \mathbb{S}^{1}(t_i)$ and each sub-walk corresponds to an $f\left(w^{t_i}, 1-\frac{\theta}{\left(\tau\|\mathbf{A}^{-1}\|\right)^2}\right)$ product where $\theta=(\theta_i)$ are the values at all the agents indexed by the walk $w^{t_i}$ and all the agents $i=1,2,\cdots,n$ appear in the walk $w^{t_i}$ at least once. Then each sub-part of the product corresponding to the walk $w^{t_i}$ is denoted as
\begin{equation*}
\Pi_{w^{t_i}}\left(1-\frac{\theta_{i(t)}}{\left(\tau\|\mathbf{A}^{-1}\|\right)^2}\right) = f\left(w^{t_i},1-\frac{\theta}{\left(\tau\|\mathbf{A}^{-1}\|\right)^2}\right)
\end{equation*} where the subscript $w^{t_i}$ denotes the consecutive product corresponding to the walk $w^{t_i}$. Furthermore each product corresponding to a $w^{t_i}\in \mathbb{S}^{1}(t_i)$ is bounded as
\begin{equation*}
f\left(w^{t_i},1-\frac{\theta}{\left(\tau\|\mathbf{A}^{-1}\|\right)^2}\right) \leqslant \Pi_{i=1}^{n} \left(1-\frac{\theta_i}{\left(\tau\|\mathbf{A}^{-1}\|\right)^2}\right)
\end{equation*} since we can always pick $n$ agents $i=1,2,\cdots,n$ in the walk $w^{t_i}$ and keep their values unchanged and let all the left $\theta_i=0$. Since $\left(1-\frac{\theta_{i(t)}}{\left(\tau\|\mathbf{A}^{-1}\|\right)^2}\right)\geqslant 0$ \eqref{>=0} and
\begin{equation*}
{\Pi_{l=1}^{n}\theta_l} \leqslant \left(\frac{1}{n}{\sum_{l=1}^{n}\theta_l}\right)^{n}
\end{equation*} holds when $\theta_l\geqslant 0$. Therefore the $\|z(t)-x^*\|^2$ is bounded as
\begin{equation*}
\begin{aligned}
& \|z(t)-x^*\|^2 \\
\leqslant & \Pi_{i=1}^{r} f\left(w^{t_i},1-\frac{\theta}{\left(\tau\|\mathbf{A}^{-1}\|\right)^2}\right) \|z(0)-x^*\|^2 \\
\leqslant & \Pi_{l=1}^{r} \Pi_{i=1}^{n}\left(1-\frac{\theta_{i}}{\left(\tau\|\mathbf{A}^{-1}\|\right)^2}\right) \|z(0)-x^*\|^2 \\
\leqslant & \Pi_{l=1}^{r} \left( \frac{1}{n} \sum_{i=1}^{n}\left( \ 1-\frac{\theta_{i}}{\left(\tau\|\mathbf{A}^{-1}\|\right)^2} \right) \right)^n \|z(0)-x^*\|^2 \\
= & \left( 1-\frac{1}{\left(\sqrt{n}\tau\|\mathbf{A}^{-1}\|\right)^2} \right)^{nr} \|z(0)-x^*\|^2,
\end{aligned}
\end{equation*} where $\sum_{{i}=1}^{n} \theta_{i} = 1$ by the Inequality Lemma \ref{lm_inequal}.

A loose bound given in terms of condition number $\kappa(\mathbf{A})=\|\mathbf{A}\|\cdot\|\mathbf{A}^{-1}\|$ is as follows. Since $\tau=\underset{i}{\mathrm{max}}\left(\|A_i\|\right)$ and $\mathbf{A}$ is full rank, then
\begin{equation*}
\tau=\underset{i}{\mathrm{max}}\left(\|A_i\|\right) < \|\mathbf{A}\|_F,
\end{equation*} where $\|\mathbf{A}\|_F=\sqrt{\sum_{i=1}^{n}\sum_{j=1}^{n}a_{ij}^{2}}$ is the Frobenius norm. The scaled condition number \citep{demmel1988probability} $\kappa_s(\mathbf{A})=\|\mathbf{A}\|_{F}\|\mathbf{A}^{-1}\|$ and the condition number $\kappa(\mathbf{A})$ satisfies the following inequality $1 \leqslant \frac{\kappa_s(\mathbf{A})}{\sqrt{n}} \leqslant \kappa(\mathbf{A})$, then
\begin{equation*}
\sqrt{n}\tau\|\mathbf{A}^{-1}\| < \sqrt{n}\|\mathbf{A}\|_F\|\mathbf{A}^{-1}\| \leqslant \kappa(\mathbf{A})
\end{equation*} and therefore the loose bound is
\begin{equation*}
\|z(t)-x^*\|^2 < \left( 1-\kappa(\mathbf{A})^{-2} \right)^{nr} \|z(0)-x^*\|^2.
\end{equation*} This concludes the proof of Theorem \ref{th_S}.
\end{proof}

\begin{remark}
The ${f}$ Bound Theorem \ref{th_S} is important since it also bounds the convergence rate of Kaczmarz's algorithm \citep{kaczmarz1937angenaherte}, which was not well solved in literature \citep{gower2015randomized}. It gives a tight bound in terms of matrix inverse $\|\mathbf{A}^{-1}\|$ and a loose bound in terms of condition number  $\kappa(\mathbf{A})$. The bounds can be easily computed when the iterative sequence of Kaczmarz's algorithm is given, compared to the known estimate \citep{galantai2005rate}. Furthermore the ${f}$ Bound Theorem \ref{th_S} clearly explains the reason that Kaczmarz's algorithm is slower than a randomized Kaczmarz's algorithm \citep{strohmer2009randomized,dai2014randomized,gower2015randomized}.
\end{remark}

\begin{remark}
With the help of ${f}^{0}$ Bound Theorem \ref{th_Sbar} and ${f}$ Bound Theorem \ref{th_S}, each product in the error updating equation \eqref{y=mu_p_y} can be divided into two parts and bounded separately. One corresponding to the ${f}$ product part where $r\geqslant 1$ and all the $1\leqslant i\leqslant n$ are present and the other one corresponding to the ${f}^{0}$ product part where not all $1\leqslant i\leqslant n$ are present. We can now prove the Bound Theorem \ref{th_bound}.
\end{remark}

\begin{proof}[Proof of Bound Theorem \ref{th_bound}]
Let $f(w^{t},\mu_{ij})$ denote the corresponding $\mu$ product of the error sequence \eqref{pppy} in \eqref{y=mu_p_y}, then according to the ${f}^{0}$ Bound Theorem \ref{th_Sbar} and ${f}$ Theorem \ref{th_S} the error updating equation \eqref{y=mu_p_y} is bounded as follows
\begin{equation*}
\begin{aligned}
& \|y_i(t+1)\| \\ 
\leqslant & \left( \sum_{j=1}^n \cdots \sum_{l_{1}=1}^n \|\mu_{il_{1}}  \cdots \mu_{l_{t-1}l_{j}} \mathbf{P}_i  \cdots  \mathbf{P}_j y_j(0)\| \right) \\
\leqslant & \sum_{j=1}^{n}\sum_{r=1}^{r_m(t)}\sum_{w_{ij}^{t}\in \mathbb{S}^{r}(t)} f(w_{ij}^{t},\frac{1}{d}) \left(1-\varphi\right)^{\frac{nr}{2}} \|y_j(0)\| \\
          & + \sum_{j=1}^{n}\sum_{w_{ij}^{t}\in \mathbb{S}^{0}(t)} f(w_{ij}^{t},\frac{1}{d}) \left(1-\varphi\right)^{\frac{0}{2}} \|y_j(0)\| \\
=         & \sum_{j=1}^{n}\sum_{r=0}^{r_m(t)}\sum_{w_{ij}^{t}\in \mathbb{S}^{r}} f(w_{ij}^{t},\frac{1}{d}) \left(1-\varphi\right)^{\frac{nr}{2}} \|y_j(0)\|
\end{aligned}
\end{equation*} where $\mu_{ij}=\frac{\alpha_{ij}}{d_i}$. Hence the proof is finished.
\end{proof}

The bound in \eqref{bound_of_th} gives another proof that the distributed algorithm studied in this paper converges to $x^*$ for connected undirected networks, as shown below.

\subsection*{Discussion of the Algorithm Convergence}
Note that the order of a $w^t$ walk typically increases as the length of walks keeps growing, since $\mathcal{G}$ is a connected network. This implies that for any given order $r$, the total number of $w^{t}\in \mathbb{S}^{r}(t)$ is limited and hence the summation of all corresponding $f(w^{t},\frac{1}{d})$ products is bounded, for the summation of all walks is 1 \eqref{sum_walk}. The number of all walks starting at vertex $v_i$ for any given order $r$ and length $t$ can be estimated by combinatorics. This method is shown when the network topology is a complete graph. For any given network, the number of walks can be bounded similarly, but it can become quite involved.

For any walk of length $t$ starting at a fixed vertex $v_0$ in a complete network $\mathcal{G}\in \mathbb{R}^{n\times n}$, the total number of all walks is $n^{t}$. Let $t\gg n$. In order to count the maximum number of $w^{t}\in \mathbb{S}^{0}(t)$ walks, we first choose subsets of vertices $\mathcal{V}_{k}^{0}\subsetneq \mathcal{V}$ by picking $k\leqslant n-2$ vertices out of $n$ and $\mathcal{V}_{n-1}^{0}\subsetneq \mathcal{V}$ by picking $n-1$ vertices except the case $v_0$ is not picked, which results in walks in $\mathbb{S}^{1}(t)$ space rather than $\mathbb{S}^{0}(t)$. There are a total $C_{n}^{k}$ of $\mathcal{V}_{k}^{0}$ sets where $C_{n}^{k}=\frac{n!}{k!\left(n-k\right)!}$, $k\leqslant n-2$ and $C_{n}^{n-1}-1$ of $\mathcal{V}_{n-1}^{0}$ sets. Then we choose a vertex with replacement each time from $\mathcal{V}_{k}^{0}$ and $\mathcal{V}_{n-1}^{0}$ and put it into the sequence of walks to generate all possible walks. The total number $c^{0}(t)$ of $w^{0}(t)$ walks is
\begin{equation*}
c^{0}(t)=\sum_{k=1}^{n-1}C_{n}^{k}k^{t}-\left(n-1\right)^{t}.
\end{equation*} The $w^{t}\in \mathbb{S}^{1}(t)$ walks are regarded as combinations of $w^{n}\in \mathbb{S}^{1}(n)$ walks and $w^{t-n}\in \mathbb{S}^{0}(t-n)$ walks. We first choose $n$ positions out of $t$ in the sequences of walks and make these $n$ positions form $w^{n}\in \mathbb{S}^{1}(n)$ walks. There are $C_{t}^{n}$ ways to choose $n$ vertices to form a $\mathcal{V}_{n}^{1}$ set and the number $w^{1}(n)$ walks is exactly $n!$ for each set, so the number of different sub-sequences in $w^{1}(t)$ walks is $P_{t}^{n}=n!C_{t}^{n}$. The number of walks $w^{0}(t-n)$ is simply $c^{0}(t-n)$. Hence the number of $w^{1}(t)$ walks is bounded by
\begin{equation*}
c^{1}(t)= P_{t}^{n}c^{0}(t-n).
\end{equation*} In general cases where $r\geqslant 2$, we pick $n$ positions out of $1,\cdots,t-(r-1)n$ locations to form the first $w^{n}\in \mathbb{S}^{1}(n)$ walk sequence and pick the left-over locations till $t-(r-2)n$ to form the second $w^{n}\in \mathbb{S}^{1}(n)$ walk and so on. Let $t_1$ denote the start position and $t_2-1$ be the end position picked out by the first $w^{n}\in \mathbb{S}^{1}(n)$ walk sequence, then the total number of sub-sequences is $P_{t_2-t_1}^{n}$. Define $t_3,\cdots t_r$ similarly, then the second $w^{n}\in \mathbb{S}^{1}(n)$ walk sequence can pick from position $t_2$ till $t_3-1$ in the original $w^{t}\in \mathbb{S}^{r}(t)$ sequence. The total number of sub-sequences for the second $w^{n}\in \mathbb{S}^{1}(n)$ walk is $P_{t_3-t_2}^{n}$. The total number of $w^{r}(t)$ walks is bounded by
\begin{equation*}
c^{r}(t)=\Pi_{i=1}^{r} P_{t_{i+1}-t_i}^{n}c^{0}\left(t-rn\right).
\end{equation*} The number of total walks then satisfies
\begin{equation*}
\lim_{t\to \infty} \frac{c^{r}(t)}{n^{t}} = \lim_{t\to \infty} \frac{t^{rn}c^{0}(t-rn)}{n^{rn}n^{t-rn}}= 0
\end{equation*} since $\frac{c^{0}(t-rn)}{n^{t-rn}}$ reduces exponentially to 0. This means for any given order $r$, the corresponding walks account for only a minor portion of all walks. In other words, the order of all products keeps growing when the length of walks increases.

Since the summation of all $f\left(w^{t},\frac{1}{d}\right)$ products is 1 \eqref{sum_walk} and the limit of the portion of walks among all walks of a given order $r_c$ is $0$, therefore the following two limits exist
\begin{equation*}
\begin{aligned}
& \lim_{t\to \infty} \rho_1=\lim_{t\to \infty} \sum_{j=1}^{n}\sum_{r=r_c+1}^{r_m(t)}\sum_{w_{ij}^{t}\in \mathbb{S}^{r}}f(w_{ij}^{t},\mu) = 1,\\
& \lim_{t\to \infty} \rho_2=\lim_{t\to \infty} \sum_{j=1}^{n}\sum_{r=0}^{r_c}\sum_{w_{ij}^{t}\in \mathbb{S}^{r}(t)}f(w_{ij}^{t},\mu) = 0.
\end{aligned}
\end{equation*} for any finite $r_c$. Furthermore, the limit of the error in \eqref{y_i} satisfies the following
\begin{equation*}
\begin{aligned}
& \lim_{t\to \infty} \|y_i(t+1)\| \\
\leqslant & \lim_{t\to \infty} \rho_1\left(1-\varphi\right)^{\frac{nr}{2}} \|y_j(0)\| +\lim_{t\to \infty} \rho_2 \left(1-\varphi\right)^{\frac{nr}{2}} \|y_j(0)\| \\ 
= & 0,
\end{aligned}
\end{equation*} where $\lim_{t\to \infty}r_m(t) = \infty$. Therefore the algorithm converges as all the $\lim_{t\to \infty}x_i(t)\to x^*$ for regular networks.

Future work will look to analyze the combinatorics of more general network topologies.

\section{CONCLUSIONS}
\label{sec_conclu}
In this work, we systematically study the impact of network topology on the performance of a network-based distributed algorithm in solving linear algebraic equations. Both theoretical analysis and simulation results show that networks with higher mean degree, smaller diameter, and more homogeneous degree distribution make the algorithm converge faster. Interestingly, $k$-regular random networks with small mean degree could have a comparable performance as degree-heterogeneous networks with very high mean degree. Hence, it is possible to reduce the communication cost (i.e. by designing sparser networks) and simultaneously keep the fast convergence rate. 

Besides classical consensus problems, we expect that more complicated problems can also be solved with network-based distributed algorithms. Our results presented here provide a method to analyse the topology impacts on a network-based distributed algorithm. It may shed light on the design of better network topologies to improve the performance of general multi-agent distributed algorithms in solving more challenging real-world problems.

\section*{Acknowledgement}
This work was partially supported by the John Templeton Foundation (award number 51977).





\bibliographystyle{IEEEtran}

\begin{thebibliography}{10}
\providecommand{\url}[1]{#1}
\csname url@rmstyle\endcsname
\providecommand{\newblock}{\relax}
\providecommand{\bibinfo}[2]{#2}
\providecommand\BIBentrySTDinterwordspacing{\spaceskip=0pt\relax}
\providecommand\BIBentryALTinterwordstretchfactor{4}
\providecommand\BIBentryALTinterwordspacing{\spaceskip=\fontdimen2\font plus
\BIBentryALTinterwordstretchfactor\fontdimen3\font minus
  \fontdimen4\font\relax}
\providecommand\BIBforeignlanguage[2]{{%
\expandafter\ifx\csname l@#1\endcsname\relax
\typeout{** WARNING: IEEEtran.bst: No hyphenation pattern has been}%
\typeout{** loaded for the language `#1'. Using the pattern for}%
\typeout{** the default language instead.}%
\else
\language=\csname l@#1\endcsname
\fi
#2}}

\bibitem{pastor2001epidemic}
R.~Pastor-Satorras and A.~Vespignani, ``Epidemic spreading in scale-free
  networks,'' \emph{Physical review letters}, vol.~86, no.~14, p. 3200, 2001.

\bibitem{cohen2000resilience}
R.~Cohen, K.~Erez, D.~Ben-Avraham, and S.~Havlin, ``Resilience of the internet
  to random breakdowns,'' \emph{Physical review letters}, vol.~85, no.~21, p.
  4626, 2000.

\bibitem{nishikawa2003heterogeneity}
T.~Nishikawa, A.~E. Motter, Y.-C. Lai, and F.~C. Hoppensteadt, ``Heterogeneity
  in oscillator networks: Are smaller worlds easier to synchronize?''
  \emph{Physical review letters}, vol.~91, no.~1, p. 014101, 2003.

\bibitem{wang2005partial}
W.~Wang and J.-J.~E. Slotine, ``On partial contraction analysis for coupled
  nonlinear oscillators,'' \emph{Biological cybernetics}, vol.~92, no.~1, pp.
  38--53, 2005.

\bibitem{liu2011controllability}
Y.-Y. Liu, J.-J. Slotine, and A.-L. Barab{\'a}si, ``Controllability of complex
  networks,'' \emph{Nature}, vol. 473, no. 7346, pp. 167--173, 2011.

\bibitem{jadbabaie2004stability}
A.~Jadbabaie, N.~Motee, and M.~Barahona, ``On the stability of the kuramoto
  model of coupled nonlinear oscillators,'' in \emph{American Control
  Conference, 2004. Proceedings of the 2004}, vol.~5.\hskip 1em plus 0.5em
  minus 0.4em\relax IEEE, 2004, pp. 4296--4301.

\bibitem{pasqualetti2014controllability}
F.~Pasqualetti, S.~Zampieri, and F.~Bullo, ``Controllability metrics,
  limitations and algorithms for complex networks,'' \emph{Control of Network
  Systems, IEEE Transactions on}, vol.~1, no.~1, pp. 40--52, 2014.

\bibitem{liu2013observability}
Y.-Y. Liu, J.-J. Slotine, and A.-L. Barab{\'a}si, ``Observability of complex
  systems,'' \emph{Proceedings of the National Academy of Sciences}, vol. 110,
  no.~7, pp. 2460--2465, 2013.

\bibitem{vicsek1995novel}
T.~Vicsek, A.~Czir{\'o}k, E.~Ben-Jacob, I.~Cohen, and O.~Shochet, ``Novel type
  of phase transition in a system of self-driven particles,'' \emph{Physical
  review letters}, vol.~75, no.~6, p. 1226, 1995.

\bibitem{jadbabaie2003coordination}
A.~Jadbabaie, J.~Lin, \emph{et~al.}, ``Coordination of groups of mobile
  autonomous agents using nearest neighbor rules,'' \emph{Automatic Control,
  IEEE Transactions on}, vol.~48, no.~6, pp. 988--1001, 2003.

\bibitem{tsitsiklis1984problems}
J.~N. Tsitsiklis, ``Problems in decentralized decision making and
  computation.'' DTIC Document, Tech. Rep., 1984.

\bibitem{tsitsiklis1986distributed}
J.~N. Tsitsiklis, D.~P. Bertsekas, M.~Athans, \emph{et~al.}, ``Distributed
  asynchronous deterministic and stochastic gradient optimization algorithms,''
  \emph{IEEE transactions on automatic control}, vol.~31, no.~9, pp. 803--812,
  1986.

\bibitem{murray2003consensus}
R.~Olfati-Saber and R.~M. Murray, ``Consensus protocols for networks of dynamic
  agents,'' in \emph{Proceedings of the 2003 American Controls Conference},
  2003.

\bibitem{olfati2007consensus}
R.~Olfati-Saber, J.~A. Fax, and R.~M. Murray, ``Consensus and cooperation in
  networked multi-agent systems,'' \emph{Proceedings of the IEEE}, vol.~95,
  no.~1, pp. 215--233, 2007.

\bibitem{mou2013fixed}
S.~Mou and A.~S. Morse, ``A fixed-neighbor, distributed algorithm for solving a
  linear algebraic equation,'' in \emph{Control Conference (ECC), 2013
  European}.\hskip 1em plus 0.5em minus 0.4em\relax IEEE, 2013, pp. 2269--2273.

\bibitem{mou2014distributed}
S.~Mou, ``Distributed control of multi-agent systems,'' Ph.D. dissertation,
  YALE UNIVERSITY, 2014.

\bibitem{liu2013asynchronous}
J.~Liu, S.~Mou, and A.~S. Morse, ``An asynchronous distributed algorithm for
  solving a linear algebraic equation,'' in \emph{Decision and Control (CDC),
  2013 IEEE 52nd Annual Conference on}.\hskip 1em plus 0.5em minus 0.4em\relax
  IEEE, 2013, pp. 5409--5414.

\bibitem{anderson2015decentralized}
B.~Anderson, S.~Mou, A.~S. Morse, and U.~Helmke, ``Decentralized gradient
  algorithm for solution of a linear equation,'' \emph{arXiv preprint
  arXiv:1509.04538}, 2015.

\bibitem{mou2015distributed}
S.~Mou, J.~Liu, and A.~S. Morse, ``A distributed algorithm for solving a linear
  algebraic equation,'' \emph{IEEE Transactions on Automatic Control}, vol.~60,
  no.~11, pp. 2863--2878, 2015.

\bibitem{neudecker1969note}
H.~Neudecker, ``A note on kronecker matrix products and matrix equation
  systems,'' \emph{SIAM Journal on Applied Mathematics}, vol.~17, no.~3, pp.
  603--606, 1969.

\bibitem{russo2013contraction}
G.~Russo, M.~Di~Bernardo, and E.~D. Sontag, ``A contraction approach to the
  hierarchical analysis and design of networked systems,'' \emph{Automatic
  Control, IEEE Transactions on}, vol.~58, no.~5, pp. 1328--1331, 2013.

\bibitem{fiedler1973algebraic}
M.~Fiedler, ``Algebraic connectivity of graphs,'' \emph{Czechoslovak
  Mathematical Journal}, vol.~23, no.~2, pp. 298--305, 1973.

\bibitem{chung1997spectral}
F.~R. Chung, \emph{Spectral graph theory}.\hskip 1em plus 0.5em minus
  0.4em\relax American Mathematical Soc., 1997, vol.~92.

\bibitem{diestel2005graph}
R.~Diestel, ``Graph theory. 2005,'' \emph{Grad. Texts in Math}, 2005.

\bibitem{watts1998collective}
D.~J. Watts and S.~H. Strogatz, ``Collective dynamics of
  ‘small-world’networks,'' \emph{nature}, vol. 393, no. 6684, pp. 440--442,
  1998.

\bibitem{barabasi1999emergence}
A.-L. Barab{\'a}si and R.~Albert, ``Emergence of scaling in random networks,''
  \emph{science}, vol. 286, no. 5439, pp. 509--512, 1999.

\bibitem{erdos1960evolution}
P.~Erd\H{o}s and A.~R{\'e}nyi, ``On the evolution of random graphs,''
  \emph{Publ. Math. Inst. Hungar. Acad. Sci}, vol.~5, pp. 17--61, 1960.

\bibitem{wormald1999models}
N.~C. Wormald, ``Models of random regular graphs,'' \emph{London Mathematical
  Society Lecture Note Series}, pp. 239--298, 1999.

\bibitem{strohmer2009randomized}
T.~Strohmer and R.~Vershynin, ``A randomized kaczmarz algorithm with
  exponential convergence,'' \emph{Journal of Fourier Analysis and
  Applications}, vol.~15, no.~2, pp. 262--278, 2009.

\bibitem{demmel1988probability}
J.~W. Demmel, ``The probability that a numerical analysis problem is
  difficult,'' \emph{Mathematics of Computation}, vol.~50, no. 182, pp.
  449--480, 1988.

\bibitem{kaczmarz1937angenaherte}
S.~Kaczmarz, ``Angen{\"a}herte aufl{\"o}sung von systemen linearer
  gleichungen,'' \emph{Bulletin International de l’Academie Polonaise des
  Sciences et des Lettres}, vol.~35, pp. 355--357, 1937.

\bibitem{gower2015randomized}
R.~M. Gower and P.~Richt{\'a}rik, ``Randomized iterative methods for linear
  systems,'' \emph{SIAM Journal on Matrix Analysis and Applications}, vol.~36,
  no.~4, pp. 1660--1690, 2015.

\bibitem{galantai2005rate}
A.~Gal{\'a}ntai, ``On the rate of convergence of the alternating projection
  method in finite dimensional spaces,'' \emph{Journal of mathematical analysis
  and applications}, vol. 310, no.~1, pp. 30--44, 2005.

\bibitem{dai2014randomized}
L.~Dai, M.~Soltanalian, and K.~Pelckmans, ``On the randomized kaczmarz
  algorithm,'' \emph{IEEE SIGNAL PROCESSING LETTERS}, vol.~21, no.~3, 2014.

\end{thebibliography}

\end{document}